\documentclass[a4paper,12pt]{article}
\usepackage{amsmath, amsthm, amscd, amsfonts, amssymb, graphicx, color}
\newtheorem{theorem}{Theorem}[section]
\newtheorem{lemma}[theorem]{Lemma}
\newtheorem{proposition}[theorem]{Proposition}
\newtheorem{corollary}[theorem]{Corollary}
\theoremstyle{definition}
\newtheorem{definition}[theorem]{Definition}

\theoremstyle{remark}

\numberwithin{equation}{section}

\newcommand{\h}{\mathcal{H}}

\begin{document}
\title{Multipliers of pg-Bessel sequences in Banach spaces}

\author{ M.R. Abdollahpour\footnote{ Department of Mathematics, Faculty of Mathematical Sciences, University of Mohaghegh Ardabili, Ardabil 56199-11367, Iran;  \ Email: {\tt mrabdollahpour@yahoo.com }},
A. Najati\footnote{ Department of Mathematics, Faculty of Mathematical Sciences, University of Mohaghegh Ardabili, Ardabil 56199-11367, Iran;  \ Email: {\tt a.nejati@yahoo.com}} and
P. G\u avru\c ta\footnote{ Department of Mathematics, Politehnica University of Timi\c soara, Pia\c ta Victoriei, Nr. 2, 300006, Timi\c soara, Romania; \ Email: {\tt pgavruta@yahoo.com }}}
{ }
\date{}
\maketitle

\begin{abstract}
In this paper, we introduce $(p,q)g-$Bessel multipliers in Banach
spaces and we show that under some conditions a $(p,q)g-$Bessel
multiplier is invertible. Also, we show the continuous dependency of
$(p,q)g-$Bessel multipliers on their parameters.
\end{abstract}



\section{Introduction}
Frames have been introduced by J. Duffin and A.C. Schaeffer in \cite{DS}, in
connection with non-harmonic Fourier series. A frame for a Hilbert space is a redundant set of vectors which yield, in a stable way, a representation for each vector in the space. The frames have many
nice properties which make them very useful in the characterization
of function space, signal processing and many other fields. See the book \cite{OCh} and the references of the paper \cite{PG}. The concept of frames was extended to Banach spaces by K. Gr\"ochenig in \cite{Gro} to develop atomic decompositions from the paper \cite{FG}. See also \cite{AlST}, \cite{Cas}, \cite{chst03}, \cite{chstheil}.
\begin{definition}
Let $X$ be a Banach space. A countable family $\{g_{i}\}_{i\in
I}\subset X^{*}$ is a $p$-frame for $X$,
$1<p<\infty$, if there exist constants $A,B>0$ such that
$$A\|f\|_{X}\leq (\sum_{i\in I}|g_{i}(f)|^{p})^{\frac{1}{p}}\leq
B\|f\|_{X},\quad f\in X.$$
\end{definition}
$G$-frame as a natural generalization of frame in Hilbert spaces,
were introduced by Sun \cite{WS} in 2006. $G$-frame cover many previous
extensions of a frame. For some properties of $g$-frames, we can
refer to \cite{ab, an, ani}.
\begin{definition}
Let $\h$ be a Hilbert space and $\{\h_{i}\}_{i\in I}$ be a sequence
of Hilbert spaces. We call a sequence $\{\Lambda_{i}\in
B(\h,\h_{i}): i\in I\}$ a $g$-frame for $\h$ with respect to
$\{\h_{i}\}_{i\in I}$ if there exist two positive constants $A$ and
$B$ such that
$$A\|f\|^{2}\leq \sum_{i\in I}\|\Lambda_{i}f\|^{2}\leq B\|f\|^{2},
\quad f\in \h.$$ We call $A$ and $B$ the lower and upper $g-$frame
bounds, respectively. We call $\{\Lambda_{i}\}_{i\in J}$ a tight
$g$-frame if $A=B$ and Parseval $g$-frame if $A=B=1.$
\end{definition}
Bessel multipliers for Hilbert spaces are investigated by Peter
Balazs \cite{xxlmult1,peter2,peter3} . We use the following
notations for sequence spaces.
\begin{enumerate}
\item[(1)] $c_0=\{\{a_n\}_{n=1}^{\infty}\subseteq \mathbb{C}:{\lim}_{n\rightarrow \infty} a_n=0\};$
\item[(2)] $l^p=\{\{a_n\}_{n=1}^{\infty}\subseteq \mathbb{C}:\|a\|_p=(\sum_{n\in \mathbb{N}}|a_n|^p)
^{\frac{1}{p}} <\infty\}, 0<p<\infty;$
\item[(3)] $l^{\infty}=\{\{a_n\}_{n=1}^{\infty}\subseteq \mathbb{C}:\|a\|_{\infty}
=\sup_{n\in \mathbb{N}}|a_n|<\infty\}.$
\end{enumerate}
\begin{definition}
Let $\h_1$ and $\h_2$ be Hilbert spaces. Let
$\{f_i\}_{i=1}^{\infty}\subseteq \h_1$ and
$\{g_i\}_{i=1}^{\infty}\subseteq \h_2$ be Bessel sequences. Fix
$m=\{m_i\}_{i=1}^{\infty} \in l^{\infty}.$ The operator
$$\mathbf{M}_{m,\{f_i\},\{g_i\}}:\h_1\rightarrow \h_2,\quad \mathbf{M}_{m,\{f_i\},\{g_i\}}
(f)=\sum_{i=1}^{\infty}m_i \langle f,f_i\rangle g_i$$ is called the
Bessel multiplier of the Bessel sequences $\{f_i\}_{i=1}^{\infty}$
and $\{g_i\}_{n=1}^{\infty}.$ The sequence $m$ is called the symbol
of $\mathbf{M}$
\end{definition}
Multipliers for $p$-Bessel sequences in Banach spaces  were introduced in
\cite{peterasghar}. Also $g$-Bessel multipliers were investigated by
Rahimi \cite{gframura}. In this note, by mixing the concepts of
multipliers for $p$-Bessel sequences and $g$-Bessel multipliers, we
will define multipliers for the $pg$-Bessel sequences ($pg-$frames)
and we will investigate some of their properties.

In our opinion, it is possible that the results of this paper can be applied in Quantum Information Theory.
A beautiful presentation of the connections between frames and POVM is the paper \cite{Ruskai}. See also \cite{Scott}.
\section{Review of pg-frames and qg-Riesz bases}
In \cite{AbFaRa}, $pg$-frames and $qg$-Riesz bases for Banach spaces
have been introduced . In this section, we recall some properties of
$pg$-frames and $qg-$Riesz bases from \cite{AbFaRa}. Throughout this
section, $I$ is a subset of $\mathbb{N}$, $X$ is a Banach space with
dual  $X^{*}$ and also $\{Y_{i}\}_{i\in I}$ is a sequence of Banach
spaces.
\begin{definition}
We call a sequence $\mathbf{\Lambda}=\{\Lambda_{i}\in B(X,Y_{i}):
i\in I\}$ a $pg-$frame for $X$ with respect to $\{Y_{i}:i\in I\}$
$(1<p<\infty),$ if there exist $A,B>0$ such that
\begin{equation}\label{pgframe}
A\|x\|_{X}\leq\left(\sum_{i\in
I}\|\Lambda_{i}x\|^p\right)^\frac{1}{p}\leq B\|x\|_{X},\quad\forall
x\in X.
\end{equation}
$A,B$ is called the $pg$-frame bounds of $\{\Lambda_{i}\}_{i\in I}.$
\\
If only the second inequality in (\ref{pgframe}) is satisfied,
$\{\Lambda_{i}\}_{i\in I}$ is called  a $pg$-Bessel sequence for $X$
with respect to $\{Y_{i}:i\in I\}$ with bound $B.$
\end{definition}
\begin{definition}
Let $\{Y_{i}\}_{i\in I}$ be a sequence of Banach spaces. We define
$$\left(\sum_{i\in I}\bigoplus Y_{i}\right)_{l_p}=\left\{\{x_i\}_{i\in I}|x_{i}\in Y_{i},
\sum_{i\in I}\|x_{i}\|^{p}<+\infty\right\}.$$ Therefore
$\left(\sum_{i\in I}\bigoplus Y_{i}\right)_{l_p}$ is a Banach space
with the norm
$$\|\{x_{i}\}_{i\in I}\|_{p}=\left(\sum_{i\in
I}\|x_{i}\|^{p}\right)^{\frac{1}{p}}.$$ Let $1<p,q<\infty$ be
conjugate exponents, i.e., $\frac{1}{p}+\frac{1}{q}=1$. If
$x^{*}=\{x_{i}^{*}\}_{i\in I}\in(\sum_{i\in I}\bigoplus
Y_{i}^{*})_{l_q}$, then one can show that the formula
$$\langle x,x^{*}\rangle=\sum_{i\in I}\langle
x_{i},x_{i}^{*}\rangle,\;\; x=\{x_{i}\}_{i\in I}\in(\sum_{i\in
I}\bigoplus Y_{i})_{l_p}$$ defines a continuous functional on
$(\sum_{i\in I}\bigoplus Y_{i})_{l_p},$ whose norm is equal to
$\|x^{*}\|_{q}.$
\end{definition}
\begin{lemma}\cite{AB}
Let $1<p,q<\infty$ such that $\frac{1}{p}+\frac{1}{q}=1,$ then
$$\left(\sum_{i\in J}\bigoplus Y_{i}\right)_{l_p}^{*}=\left(\sum_{i\in J}\bigoplus
Y_{i}^{*}\right)_{l_q};$$ where the equality holds under the
duality
$$\langle x,x^{*}\rangle=\sum_{i\in J}\langle
x_{i},x_{i}^{*}\rangle.$$
\end{lemma}
\begin{definition}\label{22}
Let $\mathbf{\Lambda}=\{\Lambda_{i}\in B(X,Y_{i}): i\in I\}$ be a
$pg$-Bessel sequence for $X$ with respect to $\{Y_{i}\}$. We define
the operators
\begin{equation}
U_{\Lambda}:X\rightarrow \left(\sum_{i\in I}\bigoplus
Y_{i}\right)_{l_p},\quad U_{\Lambda}(x)=\{\Lambda_{i}x\}_{i\in I}
\end{equation}
and
\begin{equation}\label{1}
T_{\Lambda}:\left(\sum_{i\in I}\bigoplus
Y_{i}^{*}\right)_{l_q}\rightarrow X^{*}\quad
T_{\Lambda}\{g_{i}\}_{i\in I}=\sum_{i\in I} \Lambda^{*}_{i}g_{i}.
\end{equation}

$U_{\Lambda}$ and $T_{\Lambda}$ are called the analysis and
synthesis operators of $\Lambda=\{\Lambda_{i}\}_{i\in I},$
respectively.
\end{definition}
The following proposition, characterizes the $pg$-Bessel sequence by
the operator $T_{\Lambda}$ defined in (\ref{1}).
\begin{proposition}\cite{AbFaRa}\label{Bessel}
$\mathbf{\Lambda}=\{\Lambda_{i}\in B(X,Y_{i}):i\in I\}$ is a
$pg$-Bessel sequence for $X$ with respect to $\{Y_{i}\}_{i\in I},$
if and only if the operator $T_{\Lambda}$ defined in $(\ref{1})$ is
a well defined and bounded operator. In this case, $\sum_{i\in I}
\Lambda^{*}_{i}g_{i}$ converges unconditionally for any
$\{g_i\}_{i\in I}\in \left(\sum_{i\in I}\bigoplus
Y_{i}^{*}\right)_{l_q}.$
\end{proposition}
\begin{lemma}\cite{AbFaRa}\label{gf}
If $\mathbf{\Lambda}=\{\Lambda_{i}\in B(X,Y_{i}):i\in I\}$ is a
$pg$-Bessel sequence for $X$ with respect to $\{Y_{i}\}_{i\in I},$
then
\begin{enumerate}
\item[(i)] $U_{\Lambda}^{*}=T$,
 \item[(ii)]If $\mathbf{\Lambda}=\{\Lambda_{i}\in B(X,Y_{i}):i\in I\}$ is a $pg$-frame for $X$ and
  all of $Y_{i}$'s are reflexive, then $T_{\Lambda}^{*}=U_{\Lambda}.$
\end{enumerate}
\end{lemma}
\begin{theorem}\cite{AbFaRa}\label{44}
$\mathbf{\Lambda}=\{\Lambda_{i}\in B(X,Y_{i}):i\in I\}$ is a
$pg$-frame for $X$ with respect to $\{Y_{i}\}_{i\in I}$ if and only
if $T_{\Lambda}$ defined in $(\ref{1})$ is a bounded and onto
operator.
\end{theorem}
\begin{definition}\label{qg}
Let $1<q<\infty.$ A family $\mathbf{\Lambda}=\{\Lambda_{i}\in
B(X,Y_{i}):i\in I\}$ is called a $qg-$Riesz basis for $X^{*}$ with
respect to $\{Y_{i}\}_{i\in I}$, if
\begin{enumerate}
\item[(i)] $\{f:\Lambda_{i}f=0,i\in I\}=\{0\}$ (i.e.,
$\{\Lambda_{i}\}_{i\in J}$ is $g$-complete);
\item[(ii)] There are positive
constants $A,B$ such that for any finite subset $I_{1}\subseteq I$
$$A\left(\sum_{i\in I_1}\|g_{i}\|^{q}\right)^{\frac{1}{q}}\leq \|\sum_{i\in I_{1}}\Lambda^{*}_{i}g_{i}
\|\leq B\left(\sum_{i\in
I_{1}}\|g_{i}\|^{q}\right)^\frac{1}{q},\quad g_{i}\in Y^{*}_{i}.$$
\end{enumerate}
\end{definition}
The assumptions of the definition (\ref{qg}) imply that $\sum_{i\in
J}\Lambda^{*}_{i}g_{i}$ converges unconditionally for all
$\{g_{i}\}_{i\in I}\in (\sum _{i\in I}\bigoplus Y^{*}_{i})_{l_{q}},$
 and
$$A\left(\sum_{i\in I}\|g_{i}\|^{q}\right)^{\frac{1}{q}}\leq \|\sum_{i\in I}\Lambda^{*}_{i}g_{i}
\|\leq B\left(\sum_{i\in I}\|g_{i}\|^{q}\right)^\frac{1}{q}.$$ In
\cite{AbFaRa}, it is proved that if
$\mathbf{\Lambda}=\{\Lambda_{i}\in B(X,Y_{i}):i\in I\}$ is a
$qg$-Riesz basis for $X^{*}$ with respect to $\{Y_{i}\}_{i\in I},$
then $\mathbf{\Lambda}$ is a $pg$-frame for $X$ with respect to
$\{Y_{i}\}_{i\in I}.$ Therefore $\mathbf{\Lambda}=\{\Lambda_{i}\in
B(X,Y_{i}):i\in I\}$ is a $qg$-Riesz basis for $X^*$ if and only if
the operator $T_{\Lambda}$ defined in (\ref{1}) is an invertible
operator from $\left(\sum_{i\in J}\bigoplus Y_{i}^{*}\right)_{l_q}$
onto $X^{*}$.
\begin{theorem}\cite{AbFaRa}\label{qs}
Let $\{Y_{i}\}_{i\in I}$ be a sequence of reflexive Banach spaces.
Let $\mathbf{\Lambda}=\{\Lambda_{i}\in B(X,Y_{i}):i\in I\}$ be a
$pg$-frame for $X$ with  respect to $\{Y_{i}\}_{i\in I}$. Then the
following statements are equivalent:
\begin{enumerate}
\item[(i)] $\{\Lambda_{i}\}_{i\in I}$ is a $qg$-Riesz basis for
$X^{*};$
 \item[(ii)] If $\{g_{i}\}_{i\in I}\in\left(\sum_{i\in I}\bigoplus
Y_{i}^{*}\right)_{l_{q}}$ and $\sum_{i\in I} \Lambda^{*}_{i}g_{i}=0$
then $g_{i}=0;$
 \item[(iii)] $\mathcal{R}_{U}=(\sum_{i\in I}\bigoplus
Y_{i})_{l_{p}}.$
\end{enumerate}
\end{theorem}

\section{Multipliers for pg-Bessel sequences}
In this section, we assume that $X_1$ and $X_2$ are reflexive Banach
spaces and $\{Y_i\}_{i=1}^{\infty}$ is a family of reflexive
Banach spaces. Also, we consider $p,q>1$ are real numbers such that
$\frac{1}{p}+\frac{1}{q}=1.$
\begin{proposition}\label{lamdaL}
Let $X$ be a Banach space and let $\mathbf{\Lambda}=\{\Lambda_i\in B(X,Y_i)\}_{i=1}^{\infty}$ be a
$pg$-Bessel sequence for $X$ with respect to
$\{Y_{i}\}_{i=1}^{\infty}$ with the bound $B.$
\begin{enumerate}
\item[(1)] If $\mathbf{\Theta}=\{\Theta_{i}\in B(X,Y_i)\}_{i=1}^{\infty}$ is a sequence of bounded operators such that
$\left(\sum_{i=1}^{\infty}\|\Lambda_{i}-\Theta_{i}\|^p\right)^{\frac{1}{p}}<K<\infty,$
then $\Theta$ is a $pg$-Bessel sequence for
$X$ with  bound $B+K.$
\item[(2)] Let $\mathbf{\Theta^{(n)}}=\{\Theta^{(n)}_{i}\in B(X,Y_i)\}_{i=1}^{\infty}$ be a sequence of
bounded operators such that for all $\varepsilon> 0$ there exists
$N>0$ with
$$\left(\sum_{i=1}^{\infty}\|\Lambda_{i}-\Theta^{(n)}_{i}\|^p\right)^{\frac{1}{p}}<\varepsilon,\qquad n\geq N,$$
then $\mathbf{\Theta^{(n)}}$ is a $pg$-Bessel sequence and for all
$n\geq N,$
$$\|U_{\Theta^{(n)}}-U_{\Lambda}\|\leq\varepsilon,\qquad \|T_{\Theta^{(n)}}-T_{\Lambda}\|\leq\varepsilon.$$
\end{enumerate}
\end{proposition}
\begin{proof}
(1) It is easy to show that $\sum_{i=1}^{\infty}\Theta_{i}^{*}g_i$
converges for any $\{g_i \}_{i=1}^{\infty}\in
\left(\sum_{i=1}^{\infty}\bigoplus Y_{i}^{*}\right)_{l_q}.$
Therefore, if $\{g_i \}_{i=1}^{\infty}\in
\left(\sum_{i=1}^{\infty}\bigoplus Y_{i}^{*}\right)_{l_q}$ we have
\begin{equation*}
\begin{aligned}
\|T_{\Lambda}\{g_i \}_{i=1}^{\infty}-T_{\Theta}\{g_i
\}_{i=1}^{\infty}\|=&\|\sum_{i=1}^{\infty}(\Lambda^{*}_i-\Theta^{*}_i
)g_i\|
={\sup}_{\|f\|\leq 1}\|\sum_{i=1}^{\infty}g_i(\Lambda_i f-\Theta_i
f )\|
\\
\leq&{\sup}_{\|f\|\leq 1}\sum_{i=1}^{\infty}\|g_i\|\|\Lambda_i
f-\Theta_i f \|
\\
\leq&\left(\sum_{i=1}^{\infty}\|g_{i}\|^q\right)^{\frac{1}{q}}{\sup}_{\|f\|\leq
1}\left(\sum_{i=1}^{\infty}\|
\Lambda_{i}f-\Theta_{i}f\|^p\right)^{\frac{1}{p}}
\\
\leq& K\|\{g_i \}_{i=1}^{\infty}\|_q,
\end{aligned}
\end{equation*}
and so
\begin{equation*}
\begin{aligned}
\|T_{\Theta}\{g_i \}_{i=1}^{\infty}\|\leq&\|T_{\Theta}\{g_i
\}_{i=1}^{\infty}-T_{\Lambda}\{g_i \}_{i=1}^{\infty}\|
+\|T_{\Lambda}\{g_i \}_{i=1}^{\infty}\|
\\
\leq& (B+K)\|\{g_i \} _{i=1}^{\infty}\|_q.
\end{aligned}
\end{equation*}
Consequently, Proposition \ref{Bessel} implies that $\{\Theta_i
\}_{i=1}^{\infty}$ is a $pg$-Bessel
sequence with the bound $B+K.$ \\
(2) It follows from (1) that $\{\Theta^{(n)}_i \}_{i=1}^{\infty}$ is
a $pg$-Bessel sequence  and
$\|T_{\Theta^{(n)}}-T_{\Lambda}\|\leq\varepsilon$ for all $n\geq N$.
But for $f\in X$ and $n\geq N$ we have
$$\|U_{\Lambda}f-U_{\Theta^{(n)}}f\|_{p}=\left(\sum_{i=1}^{\infty}\|\Lambda_{i}f-\Theta^{(n)}_{i}f\|^p\right)^{\frac{1}
{p}}\leq
 \left(\sum_{i=1}^{\infty}\|\Lambda_{i}-\Theta^{(n)}_{i}\|^p\right)^{\frac{1}{p}}\|f\|$$ hence $\|U_{\Theta^{(n)}}-U_
 {\Lambda}
 \|\leq\varepsilon.$
\end{proof}
\begin{proposition}\label{mul}
Let $\mathbf{\Lambda}=\{\Lambda_{i}\in
B(X_2,Y_{i})\}_{i=1}^{\infty}$ be a $pg$-Bessel sequence for $X_2$
with bound $B_{\Lambda}$ and $\mathbf{\Theta}=\{\Theta_{i}\in
B(X_1^*,Y_{i}^{*})\}_{i=1}^{\infty}$ be a $qg$-Bessel sequence for
$X_1^*$ with bound $B_{\Theta}$. If $m\in l^{\infty},$ then the
operator
$$\mathbf{M}_{m,\Lambda,\Theta}:X_1^*\rightarrow
X_2^*,\qquad\mathbf{M}
_{m,\Lambda,\Theta}(g)=\sum_{i=1}^{\infty}m_i\Lambda_i^*\Theta_i g$$
is well defined, the sum converges unconditionally for all $g\in
X_{1}^{*}$ and $$\|\mathbf{M}_{m,\Lambda,\Theta}\|\leq
B_{\Lambda}B_{\Theta}\|m\|_{\infty}.$$
\end{proposition}
\begin{proof}
Let $g\in X_{1}^{*}$, then $\{m_i \Theta_i g\}_{i=1}^{\infty}\in
\left(\sum_{i=1}^{\infty}\bigoplus Y_{i}^{*}\right)_ {l_q},$ and
Proposition \ref{Bessel} implies that
$\sum_{i=1}^{\infty}m_i\Lambda_i^*\Theta_i g$ converges
unconditionally and $\mathbf{M}_{m,\Lambda,\Theta}$ is well defined. Also we have
\begin{equation*}
\begin{aligned}
\|\sum_{i=1}^{\infty}m_i\Lambda_i^*\Theta_i
g\|=&{\sup}_{\|x\|\leq1}|\langle x,\sum_{i=1}^{\infty}m_i\Lambda_i^*\Theta_i
g\rangle|
={\sup}_{\|x\|\leq1}|\sum_{i=1}^{\infty}m_i(\Theta_i g)(\Lambda_i x)|
\\
\leq&{\sup}_{\|x\|\leq1}\sum_{i=1}^{\infty}|m_i||(\Theta_i
g)(\Lambda_i x)|
\leq\|m\|_{\infty}{\sup}_{\|x\|\leq1}\sum_{i=1}^{\infty}\|\Theta_i
g\|\|\Lambda_i x\|
\\
\leq&\|m\|_{\infty}\left(\sum_{i=1}^{\infty}\|\Theta_{i}g\|^q\right)^{\frac{1}{q}}{\sup}_{\|x\|\leq
1}\left(\sum_ {i=1}^{\infty}\|\Lambda_{i}x\|^p\right)^{\frac{1}{p}}
\\
\leq&\|m\|_{\infty}\cdot B_{\Theta}\|g\|\cdot{\sup}_{\|x\|\leq
1}(B_{\Lambda}\|x\|)
\\
\leq&\|m\|_{\infty}\cdot B_{\Theta}\cdot B_{\Lambda}\|g\|.
\end{aligned}
\end{equation*}
Therefore $\mathbf{M}_{m,\Lambda,\Theta}$ is bounded and $\|\mathbf{M}_{m,\Lambda,\Theta}\|\leq B_{\Lambda}
B_{\Theta}\|m\|_{\infty}$.
\end{proof}
\begin{definition}
Let $\mathbf{\Lambda}=\{\Lambda_{i}\in
B(X_2,Y_{i})\}_{i=1}^{\infty}$ be a $pg$-Bessel sequence for $X_2$
with bound $B_{\Lambda}$ and $\mathbf{\Theta}=\{\Theta_{i}\in
B(X_1^*,Y_{i}^{*})\}_{i=1}^{\infty}$ be a $qg$-Bessel sequence for
$X_1^*$ with bound $B_{\Theta}$. Let $m=\{m_i\}_{i=1}^{\infty}\in
l^{\infty}.$ The operator
\begin{equation}
\mathbf{M}_{m,\Lambda,\Theta}:X_1^*\rightarrow X_2^*,\qquad\mathbf{M}_{m,\Lambda,\Theta}(g)=\sum_{i=1}^{\infty}m_i
\Lambda_i^*\Theta_i g
\end{equation}
is called the $(p,q)g-$Bessel multiplier of
$\mathbf{\Lambda}$, $ \mathbf{\Theta}$ and $m$. The sequence $m$ is called the symbol of $ \mathbf{M}.$
\end{definition}
\begin{proposition}
Let $\mathbf{\Lambda}=\{\Lambda_{i}\in
B(X_2,Y_{i})\}_{i=1}^{\infty}$ be a $qg$-Riesz basis for $X^*_2$ and
$\mathbf{\Theta}=\{\Theta_{i}\in
B(X_1^*,Y_{i}^{*})\}_{i=1}^{\infty}$ be a $qg$-Bessel sequence for
$X_1^*$ with non zero members. Then the mapping
$$
m\rightarrow \mathbf{M}_{m,\Lambda,\Theta}
$$
is injective from $l^{\infty}$ into $B(X^*_1, X^*_2)$.
\end{proposition}
\begin{proof}
If $\mathbf{M}_{m,\Lambda,\Theta}=0$,  then
$\sum_{i=1}^{\infty}m_i\Lambda_i^*\Theta_i g=0$ for all $g\in
X_1^*$. Then Theorem \ref{qs} implies that $m_i\Theta_i g=0$ for all
$i\in\mathbb{N}$ and for all $g\in X_1^*$. Since $\Theta_i \neq 0$
for each
 $i\in\mathbb{N},$ we get $m_i=0.$
\end{proof}
\begin{theorem}\label{dual1}
Let $\mathbf{\Lambda}=\{\Lambda_{i}\in
B(X_2,Y_{i})\}_{i=1}^{\infty}$ be a $qg$-Riesz basis for $X^{*}_2$
with respect to $\{Y_{i}\}_{i=1}^{\infty}$, then there exist a
sequence $\{\widetilde{\Lambda}_i\in B(X^{*}_2,
Y^{*}_{i})\}_{i=1}^{\infty}$ which is a $pg$-Riesz basis for $X_2$
with respect to $\{Y^{*}_{i}\}_{i=1}^{\infty}$ such that
$$x^*=\sum_{i=1}^{\infty}\Lambda^{*}_{i}\widetilde{\Lambda}_i
x^*,\qquad x^*\in X^{*}_2$$ and $\widetilde{\Lambda}_k
\Lambda^{*}_{i}=\delta_{k,i}I.$
\end{theorem}
\begin{proof}
Since $\mathbf{\Lambda}=\{\Lambda_{i}\in
B(X_2,Y_{i})\}_{i=1}^{\infty}$ is a $pg$-frame for $X_2,$ Theorem
\ref{44} implies that for every $x^*\in X^{*}_2$ there exists
$\{g_i\}_{i=1}^{\infty}\in\left(\sum_{i=1}^ {\infty}\bigoplus
Y_{i}^{*}\right)_{l_q}$  such that $x^*=\sum_{i=1}^{\infty}
\Lambda^{*}_{i} g_i.$ Let us define the operator
$$\widetilde{\Lambda}_i:X^{*}_2\rightarrow Y_{i}^{*},\qquad \widetilde{\Lambda}_i(x^*)=g_i.$$
By Theorem \ref{qs}, $\widetilde{\Lambda}_i$ is well defined. Let
$A_{\Lambda},B_{\Lambda}$ be the $qg$-Riesz basis bounds for
$\mathbf{\Lambda}=\{\Lambda_{i}\in B(X_2,Y_{i})\}_{i=1}^{\infty}$.
Then for any
$\{g_i\}_{i=1}^{\infty}\in\left(\sum_{i=1}^{\infty}\bigoplus
Y_{i}^{*}\right)_{l_q}$ we have
\[A_{\Lambda}\left(\sum_{i=1}^{\infty}\|g_{i}\|^q\right)^{\frac{l}{q}}\leq \|\sum_{i=1}^{\infty}\Lambda^{*}_{i}g_i\|\leq B_{\Lambda}\left(\sum_{i=1}^{\infty}\|g_{i}\|^q\right)^{\frac{l}{q}}.\] Therefore
$$\frac{1}{B_{\Lambda}}\|\sum_{i=1}^{\infty} \Lambda^{*}_{i}g_i\|\leq \left(\sum_{i=1}
^{\infty}\|g_{i}\|^q\right)^{\frac{l}{q}}\leq \frac{1}{A_{\Lambda}}\|\sum_{i=1}^{\infty}\Lambda^{*}_{i}g_i\|, $$
for all $\{g_i\}_{i=1}^{\infty}\in\left(\sum_{i=1}^{\infty}\bigoplus Y_{i}^{*}\right)_{l_q}.$ Hence we get
$$\frac{1}{B_{\Lambda}}\|x^*\|\leq \left(\sum_{i=1}^{\infty}\|\widetilde{\Lambda}_i(x^*)\|^q\right)^
{\frac{l}{q}}\leq \frac{1}{A_{\Lambda}}\|x^*\|,\quad x^*\in X_2^*.
$$ This implies that $\{\widetilde{\Lambda}_i\in
B(X^{*}_2,Y^{*}_{i})\}_{i=1}^{\infty}$ is a $qg$-frame for $X^{*}_2$
with respect to $\{Y^{*}_{i}\}_{i=1}^{\infty}$ with bounds
$\frac{1}{A_{\Lambda}}$ and $\frac{1}{B_{\Lambda}}$ and
$$x^*=\sum_{i=1}^{\infty}\Lambda^{*}_{i}\widetilde{\Lambda}_i
x^*,\quad x^*\in X^{*}_2$$ and $\widetilde{\Lambda}_k
\Lambda^{*}_{i} =\delta_{k,i}I.$ From other hand the synthesis
operator is invertible and
$U_{\widetilde{\Lambda}}=T_{\Lambda}^{-1},$ therefore
$U_{\widetilde{\Lambda}}$ is invertible. So by lemma \ref{gf},
$U^{*}_{\widetilde{\Lambda}}=T_{\widetilde{\Lambda}}$ is invertible
and therefore $\{\widetilde{\Lambda}_{i}\}_{i\in \mathbb{N}}$ is a
$pg$-Riesz basis for $X_2.$
\end{proof}
\begin{corollary}\label{dual2}
Let $\mathbf{\Theta}=\{\Theta_{i}\in
B(X^{*}_{1},Y^{*}_{i})\}_{i=1}^{\infty}$ be a $pg$-Riesz basis for
$X_1$ with respect to $\{Y^{*}_{i}\}_{i=1}^{\infty}$ with bounds
$A_{\Theta},B_{\Theta}$, then there exists a sequence
$\{\widetilde{\Theta}_i\in B(X_1,Y_{i})\}_{i=1}^{\infty}$ which is a
$qg$-Riesz basis for $X^{*}_1$ with respect to
$\{Y_{i}\}_{i=1}^{\infty}$ with bounds
$\frac{1}{B_{\Theta}},\frac{1}{A_{\Theta}}$ and
$$x=\sum_{i=1}^{\infty}\Theta^{*}_{i}\widetilde{\Theta}_i x,\qquad x\in X_1,$$ and
$\widetilde{\Theta}_k \Theta^{*}_{i}=\delta_{k,i}I.$
\end{corollary}
\begin{proposition}
Let $\mathbf{\Lambda}=\{\Lambda_{i}\in
B(X_2,Y_{i})\}_{i=1}^{\infty}$ be a $qg$-Riesz basis  for
$X^{*}_{2}$ with respect to $\{Y_{i}\}_{i=1}^{\infty}$ with bound
$A_{\Lambda}, B_{\Lambda}$ and $\mathbf{\Theta}=\{\Theta_{i}\in
B(X^{*}_{1} ,Y^{*}_{i})\}_{i=1}^{\infty}$ be a $pg$-Riesz basis for
$X_1$ with respect to $\{Y^{*}_{i}\}_{i=1}^{\infty}$ with bounds
$A_{\Theta}, B_{\Theta}$. If $m\in l^{\infty},$ then
$$A_{\Lambda}A_{\Theta} \|m\|_{\infty}\leq\|\mathbf{M}_{m,\Lambda,\Theta}\|\leq B_{\Lambda}B_{\Theta} \|m\|_{\infty}.$$
\end{proposition}
\begin{proof}
By proposition \ref{mul}, it is enough to show that we have the
lower bound. Corollary \ref{dual2} implies that there exists a
sequence $\{\widetilde{\Theta}_i\in B(X_1,Y_{i})\}_{i=1}^{\infty}$
which is a $qg$-Riesz basis  for $X^{*}_{1}$ (therefore a $pg$-frame
for $X_1$) with respect to $\{Y_{i}\}_{i=}^{\infty}$  with bounds
$\frac{1}{B_{\Theta}},\frac{1}{A_{\Theta}}$ and
$$x=\sum_{i=1}^{\infty}\Theta^{*}_{i}\widetilde {\Theta}_i x,\qquad
x\in X_1,$$ and $\widetilde{\Theta}_k \Theta^{*}_{i}=\delta_{k,i}I.$
Let us fix $0\neq y_k^*\in Y_k^*$ for each $k\in\mathbb{N},$ then we
have
\begin{equation*}
\begin{aligned}
\|\mathbf{M}_{m,\Lambda,\Theta}\|=&\sup_{0\neq g\in X_{1}^* }\frac{\|\mathbf{M}_{m,\Lambda,\Theta}g\|}{\|g\|}=\sup_{0
\neq g\in X_{1}^* }\frac{\|\sum_{i=1}^{\infty}m_i\Lambda_i^*\Theta_i g\|}{\|g\|}
\\
\geq&\sup_{k\in\mathbb{N}
}\frac{\|\sum_{i=1}^{\infty}m_i\Lambda_{i}^{*}\Theta_i
(\widetilde{\Theta}_ {k})^{*}y_k^*
\|}{\|(\widetilde{\Theta}_{k})^{*}y_k^* \|}
\\
=&\sup_{k\in\mathbb{N}}\frac{\|m_k \Lambda_k^* y_k^*
\|}{\|(\widetilde{\Theta}_{k})^{*}y_k^* \|}
\\
=&\sup_{k\in\mathbb{N} }|m_k |\frac{\|\Lambda_k^* y_k^*
\|}{\|(\widetilde{\Theta}_{k})^{*}y_k^* \|}
\\
\geq& A_{\Lambda}A_{\Theta}\|m\|_{\infty}.
\end{aligned}
\end{equation*}
\end{proof}
\begin{theorem}
Let $\mathbf{\Lambda}=\{\Lambda_{i}\in
B(X_2,Y_{i})\}_{i=1}^{\infty}$ be a $qg$-Riesz basis for $X^{*}_2$
with respect to $\{Y_{i}\}_{i=1}^{\infty}$ and
$\mathbf{\Theta}=\{\Theta_{i}\in
B(X^{*}_{1},Y^{*}_{i})\}_{i=1}^{\infty}$ be a $pg$-Riesz basis for
$X_1$ with respect to $\{Y^{*}_{i}\}_{i=1}^{\infty}$. If
$m=\{m_i\}_{i=1}^{\infty}$
 satisfies  $0<{\inf}_{i\in\mathbb{N}} |m_i|\leq {\sup}_{i\in\mathbb{N}} |m_i|<+\infty,$ then
 $\mathbf{M}_{m,\Lambda,\Theta}$ is invertible.
\end{theorem}
\begin{proof}

Let us consider $\{\widetilde{\Lambda}_i\in
B(X^{*}_2,Y^{*}_{i})\}_{i=1}^{\infty}$ and
$\{\widetilde{\Theta}_i\in B(X_1,Y_{i})\}_{i=1}^{\infty}$ which
appear in Proposition \ref{dual1} and Corollary \ref{dual2},
respectively. We prove that
$$(\mathbf{M}_{m,\Lambda,\Theta})^{-1}=\mathbf{M}_{\frac{1}{m},\widetilde{\Theta},\widetilde{\Lambda}}.$$ Let
$ g\in X_{1}^*$, then
\begin{equation*}
\begin{aligned}
\mathbf{M}_{\frac{1}{m},\widetilde{\Theta},\widetilde{\Lambda}}\circ
\mathbf{M}_{m,\Lambda,\Theta}(g)=&\mathbf{M}
_{\frac{1}{m},\widetilde{\Theta},\widetilde{\Lambda}}\big(\sum_{i=1}^{\infty}m_i\Lambda_i^*\Theta_i
g\big)
\\
=&\sum_{k=1}^{\infty}\frac{1}{m_k}(\widetilde{\Theta}_{k})^{*}\widetilde{\Lambda}_k \big(\sum_{i=1}^{\infty}
m_i\Lambda_i^*\Theta_i g\big)
\\
=&\sum_{k=1}^{\infty}\frac{1}{m_k}(\widetilde{\Theta}_{k})^{*} \big(\sum_{i=1}^{\infty}m_i\widetilde{\Lambda}
_k\Lambda_i^*\Theta_i g\big)
\\
=&\sum_{k=1}^{\infty}\frac{1}{m_k}(\widetilde{\Theta}_{k})^{*} (m_k\Theta_k g)
\\
=&g.
\end{aligned}
\end{equation*}
Let us consider $ f\in X_{2}^*$, then
\begin{equation*}
\begin{aligned}
\mathbf{M}_{m,\Lambda,\Theta}\circ \mathbf{M}_{\frac{1}{m},\widetilde{\Theta},\widetilde{\Lambda}}f =&\mathbf{M}_{m,
\Lambda,\Theta}\big(\sum_{k=1}^{\infty}\frac{1}{m_k}(\widetilde{\Theta}_{k})^{*}\widetilde{\Lambda}_k f\big)
\\
=&\sum_{i=1}^{\infty}m_i\Lambda_i^*\Theta_i \big(\sum_{k=1}^{\infty}\frac{1}{m_k}(\widetilde{\Theta}_{k})^{*}
\widetilde{\Lambda}_k f\big)
\\
=&\sum_{i=1}^{\infty}m_i\Lambda_i^* \big(\sum_{k=1}^{\infty}\frac{1}{m_k}\Theta_i(\widetilde{\Theta}_{k})^{*}
\widetilde{\Lambda}_k f\big)
\\
=&\sum_{i=1}^{\infty}m_i\Lambda_i^* (\frac{1}{m_i}\widetilde{\Lambda}_i f)
\\
=&f.
\end{aligned}
\end{equation*}
\end{proof}

In the next results, we show that the $(p,q)g$-Bessel multiplier
$\mathbf{M}=\mathbf{M}_{m,\Lambda,\Theta}$ depends continuously on
its parameters, $m=\{m_{i}\}_{i=1}^{\infty},$
$\mathbf{\Lambda}=\{\Lambda_{i}\}_{i=1}^{\infty}$ and
$\mathbf{\Theta}=\{\Theta_{i}\}_{i=1}^{\infty}.$

\begin{theorem}\label{depen}
Let $\mathbf{\Lambda}=\{\Lambda_{i}\in
B(X_2,Y_{i})\}_{i=1}^{\infty}$ be a $pg$-Bessel sequence for $X_2$
with bound $B_{\Lambda}$ and $\mathbf{\Theta}=\{\Theta_{i}\in
B(X_1^*,Y_{i}^{*})\}_{i=1}^{\infty}$ be a $qg$-Bessel sequence
 for $X_1^*$ with bound $B_{\Theta}$.  Let $p_1,q_1>1$
such that $\frac{1}{p_1}+\frac{1}{q_1}=1$ and $m\in l^{\infty}.$ Let
$\mathbf{\Lambda^{(n)}}=\{\Lambda^{(n)}_{i}\in
B(X_2,Y_{i})\}_{i=1}^{\infty}$ be a $pg$-Bessel sequence for $X_2$
with bound $B_{\Lambda^{(n)}}$ and
$\mathbf{\Theta^{(n)}}=\{\Theta^{(n)}_{i}\in
B(X_1^*,Y_{i}^{*})\}_{i=1}^{\infty}$ be a $qg$-Bessel sequence for
$X_1^*$ with bound $B_{\Theta^{(n)}}$ for all $n\in \mathbb{N}.$

Then
\begin{enumerate}
\item[(1)] If $\|m^{(n)}-m\|_{p_1}\rightarrow 0,$ then
$\|\mathbf{M}_{m^{(n)},\Lambda,\Theta}-\mathbf{M}_{m,\Lambda,\Theta}\|\rightarrow
0,$ as $n\rightarrow \infty.$
\item[(2)]If $m\in l^{p_1}$ and $\{\Theta^{(n)}_{i}\}_{i=1}^{\infty}$ converges to $\{\Theta_{i}\}_{i=1}^{\infty}$ in
$l^{q_1}$-sense, then
$$\|\mathbf{M}_{m,\Lambda,\Theta^{(n)}}-\mathbf{M}_{m,\Lambda,\Theta}\|\rightarrow
0,\quad n\rightarrow\infty. $$
\item[(3)] If $m\in l^{p_1}$ and $\{\Lambda^{(n)}_{i}\}_{i=1}^{\infty}$ converges to $\{\Lambda_{i}\}_{i=1}^{\infty}$ in
$l^{q_1}$-sense, then
$$\|\mathbf{M}_{m,\Lambda^{(n)},\Theta}-\mathbf{M}_{m,\Lambda,\Theta}\|\rightarrow
0,\quad n\rightarrow\infty. $$
\item[(4)] Let
$$B_1=\sup_{n\in \mathbb{N}}B_{\Lambda^{(n)}}<+\infty,\quad B_2=\sup_{n\in \mathbb{N}}B_{\Theta^{(n)}}<+\infty.$$
If $\|m^{(n)}-m\|_{l^{p_1}}\rightarrow 0$ and $\{\Theta^{(n)}_{i}\}_{i=1}^{\infty}$
and $\{\Lambda^{(n)}_{i}\}_{i=1}^{\infty}$ converge
to $\{\Theta_{i}\}_{i=1}^{\infty}$ and $\{\Lambda_{i}\}_{i=1}^{\infty}$ in $l^{q_1}$-sense, respectively, then
$$\|\mathbf{M}_{m^{(n)},\Lambda^{(n)},\Theta^{(n)}}-\mathbf{M}_{m,\Lambda,\Theta}\|\rightarrow
0,\quad n\rightarrow\infty. $$
\end{enumerate}
\end{theorem}
\begin{proof}
(1) Using proof of the Proposition \ref{mul} we have
\begin{equation*}
\begin{aligned}
\|\mathbf{M}_{m^{(n)},\Lambda,\Theta}-\mathbf{M}_{m,\Lambda,\Theta}\|=&\|\mathbf{M}_{m^{(n)}-m,\Lambda,\Theta}\|
\\
\leq& B_{\Lambda} B_{\Theta}\|m^{(n)}-m\|_{\infty}
\\
\leq& B_{\Lambda} B_{\Theta}\|m^{(n)}-m\|_{p_1}\rightarrow 0.
\end{aligned}
\end{equation*}
\\
(2) For $g\in X_1^*,$ we have
 \begin{equation*}
\begin{aligned}
\|\mathbf{M}_{m,\Lambda,\Theta^{(n)}}g-\mathbf{M}_{m,\Lambda,\Theta}g\|=&\|\sum_{i=1}^{\infty}m_i\Lambda_i^*(\Theta^
{(n)}_i-\Theta_i)g\|
\\
\leq&\sum_{i=1}^{\infty}|m_i|\|\Lambda_i^*\|\|(\Theta^{(n)}_i-\Theta_i)g\|
\\
\leq&\sum_{i=1}^{\infty}B_\Lambda|m_i|\|(\Theta^{(n)}_i-\Theta_i)g\|
\\
\leq&B_\Lambda\|m\|_{p_1}\left(\sum_{i=1}^{\infty}\|(\Theta^{(n)}_i-\Theta_i)g\|^{q_1}\right)^{\frac{1}{q_1}}.
\end{aligned}
\end{equation*}
So
$$\|\mathbf{M}_{m,\Lambda,\Theta^{(n)}}-\mathbf{M}_{m,\Lambda,\Theta}\|\leq B_\Lambda\|m\|_{p_1}\left(
\sum_{i=1}^{\infty}\|(\Theta^{(n)}_i-\Theta_i)\|^{q_1}\right)^{\frac{1}{q_1}}\rightarrow
0.$$ (3) It is similar to the proof of (2). \\
(4) We have
\begin{equation}\label{l2}
\|\mathbf{M}_{m^{(n)},\Lambda^{(n)},\Theta^{(n)}}-\mathbf{M}_{m,\Lambda^{(n)},\Theta^{(n)}}\|\leq B_1 B_2\|m^{(n)}-m\|_{p_1},
\end{equation}
\begin{equation}\label{l3}
\|\mathbf{M}_{m,\Lambda^{(n)},\Theta^{(n)}}-\mathbf{M}_{m,\Lambda,\Theta^{(n)}}\|\leq
B_2\|m\|_{p_1}\left(\sum_{i=1}^{\infty}\|(\Lambda^{(n)}_i-\Lambda_i)\|^{q_1}\right)^{\frac{1}{q_1}},
\end{equation}
\begin{equation}\label{l4}
\|\mathbf{M}_{m,\Lambda,\Theta^{(n)}}-\mathbf{M}_{m,\Lambda,\Theta}\|\leq
B_\Lambda\|m\|_{p}\left(
\sum_{i=1}^{\infty}\|(\Theta^{(n)}_i-\Theta_i)\|^{q_1}\right)^{\frac{1}{q_1}}.
\end{equation}
Since
\begin{equation*}\label{l1}
\begin{aligned}
\|\mathbf{M}_{m^{(n)},\Lambda^{(n)},\Theta^{(n)}}-\mathbf{M}_{m,\Lambda,\Theta}\|\leq&\|\mathbf{M}_{m^{(n)},\Lambda^{(n)}
,\Theta^{(n)}}-\mathbf{M}_{m,\Lambda^{(n)},\Theta^{(n)}}\|
\\
+&\|\mathbf{M}_{m,\Lambda^{(n)},\Theta^{(n)}}-\mathbf{M}_{m,\Lambda,\Theta^{(n)}}\|
\\
+&\|\mathbf{M}_{m,\Lambda,\Theta^{(n)}}-\mathbf{M}_{m,\Lambda,\Theta}\|,
\end{aligned}
\end{equation*}
  (\ref{l2}), (\ref{l3}),
 (\ref{l4}) imply that
 $$\|\mathbf{M}_{\Lambda^{(n)},\Theta^{(n)},m^{(n)}}-\mathbf{M}_{\Lambda,\Theta,m}\|\rightarrow
0,\quad n\rightarrow\infty. $$
\end{proof}

{\bf Acknowledgements:} The work of P. G\u avru\c ta was partial supported by a grant of Romanian National Authority for Scientific Research, CNCS-UEFISCDI, project number PN-II-ID-JRP-RO-FR-2011-2-11-RO-FR/01.03.2013.\\

\bibliographystyle{amsplain}

\end{document}